\documentclass[11pt,a4paper]{amsart}

\usepackage[T1]{fontenc}
\usepackage[utf8]{inputenc}
\usepackage{amssymb,amsmath,amsfonts,amsthm,mathtools}
\usepackage{enumitem}
\usepackage{hyperref}
\usepackage{nicefrac}

\newtheorem{theorem}{Theorem}[section]
\newtheorem{lemma}[theorem]{Lemma}
\newtheorem{proposition}[theorem]{Proposition}
\newtheorem{corollary}[theorem]{Corollary}

\theoremstyle{definition}
\newtheorem{definition}[theorem]{Definition}
\newtheorem{example}[theorem]{Example}
\newtheorem{remark}[theorem]{Remark}

\DeclareMathOperator{\arcosh}{arcosh}
\DeclareMathOperator{\tr}{tr}
\DeclareMathOperator{\vol}{vol}
\DeclareMathOperator{\Isom}{Isom}

\newcommand{\bH}{\mathbb{H}}
\newcommand{\R}{\mathbb{R}}
\newcommand{\Q}{\mathbb{Q}}
\newcommand{\Z}{\mathbb{Z}}
\newcommand{\C}{\mathbb{C}}

\newcommand{\cL}{\mathcal{L}}
\newcommand{\cO}{\mathcal{O}}
\newcommand{\cN}{\mathcal{N}}
\newcommand{\cM}{\mathcal{M}}

\newcommand{\cH}{\mathcal{H}}
\newcommand{\PSL}{\mathrm{PSL}}
\newcommand{\SL}{\mathrm{SL}}
\newcommand{\id}{\mathrm{id}}
\newcommand{\rM}{\mathrm{M}}
\newcommand{\rP}{\mathrm{P}}
\newcommand{\norm}[1]{\lVert #1 \rVert}
\newcommand{\abs}[1]{\lvert #1 \rvert}
\newcommand{\NK}{N_{K\mid\Q}}

\newcommand{\VLocal}{\mathbb{V}}

\title[Mean multiplicities for semi-arithmetic surfaces]{
  Exponential Growth of Mean Multiplicities in Length Spectra\\
  of Semi-Arithmetic Surfaces of Higher Arithmetic Dimension}

\author[A. Zuevsky]{A. Zuevsky}
\address{Institute of Mathematics, Czech Academy of Sciences, \v{Z}itn\'{a} 25, Czech Republic}
\email{zuevsky@yahoo.com}


\begin{document}

\begin{abstract}
For a semi-arithmetic Fuchsian group $\Gamma$ of arithmetic dimension $r$
with a generalized modular embedding, EGMM holds for
$r\leq2$, and for $r\geq3$ under an explicit \emph{strong contraction
condition}, via a multi-dimensional Schwarz-Pick bound and a
geometry-of-numbers estimate. A covering construction turns every known
example into an infinite explicit family with the same exponent via
principal congruence towers. We add new $r\leq2$ examples (genus-$3$/$4$
Prym eigenform loci, a different commensurability class from the known
triangle groups and McMullen genus-two family) and attempt a direct
flat-geometric bound on the contraction constant $\delta$ for the latter
via its Lyapunov exponents ($\delta=2/3$ on ergodic average, by
Bainbridge exact computation), explaining why an average
bound does not yet certify the pointwise one \eqref{eqstrong-contraction}
needs. 
\end{abstract}

\maketitle
\tableofcontents

\section{Introduction}
\label{secintro}

Aurich and Steiner  numerical discovery of anomalously large multiplicities
in the geodesic length spectrum of certain hyperbolic surfaces \cite{AuSt88}
was later tied to arithmeticity: Bogomolny-Georgeot-Giannoni-Schmit
\cite{Bog92} observed that exponential growth of mean multiplicities (EGMM)
tends to occur for arithmetic Fuchsian groups, and Luo-Sarnak
\cite{LuoSarnak94} and Bolte \cite{Bol93} made this rigorous via the bounded
clustering (B-C) property, later shown by Geninska-Leuzinger \cite{GenLeu08}
to characterize arithmeticity among non-cocompact groups. Whether EGMM could
occur outside the arithmetic setting remained open until
Belolipetsky-Cosac-D\'{o}ria-Teixeira Paula \cite{BCDT25mult} proved it
for semi-arithmetic groups of arithmetic dimension $r\leq2$ admitting a
modular embedding, using a Schwarz-Pick contraction argument in place of
B-C. They asked what happens for $r\geq3$. This is the question we take up.

\begin{definition}
\label{defsemi-arith-intro}
A finite-covolume Fuchsian group $\Gamma$ is \emph{semi-arithmetic} of
\emph{arithmetic dimension} $r$ if its invariant trace field $K$ (degree
$d=[K:\Q]$) is totally real, traces of $\Gamma$ are algebraic integers, and
exactly $r$ of the $d$ real embeddings of $K$ carry $\tr(\Gamma^{(2)})$ to an
unbounded subset of $\R$ (necessarily including the identity embedding), see
 Section \ref{secprelim}. Arithmetic groups are exactly the case $r=1$.
\end{definition}

For $r\geq2$, non-arithmeticity makes the Galois-conjugate representations
to be inequivalent to the identity one, and when a compatible
holomorphic \emph{generalized modular embedding} $F:\bH\to\bH^{r-1}$ exists
(Definition \ref{defgen-modular}), $F$ is a strict contraction of the
product hyperbolic metric by the Schwarz-Pick lemma. Our contribution is
twofold: a bound on how the $r-1$ separate contractions combine, and a
sharper geometry-of-numbers argument turning that bound into a length-spectrum
estimate.

\subsection*{Main results}
Set $\delta=\delta(\Gamma)\in(0,1)$ for the Schwarz-Pick contraction
constant of $F$ (Lemma \ref{lemmulti-SP}), well defined by compactness of a
fundamental domain in the cocompact case, and by a compact-core argument
otherwise (Lemma \ref{lemnoncompact-displ}). 
   We stress that \emph{the uniform, everywhere-positive
form of this statement is established here first for compact groups}, where
the supremum defining $\delta$ is attained on a compact fundamental domain. 
 The non-cocompact case is treated separately in Section \ref{secnoncompact} via
the compact core, and needs its argument (Lemma \ref{lemnoncompact-displ}),
not a direct limiting appeal to the cocompact case.

For $r\geq3$ define the \emph{strong contraction condition}
\begin{equation}
\label{eqstrong-contraction}
\delta(\Gamma) \;>\; 1-\frac{1}{\sqrt2\,(r-1)}.
\end{equation}

\begin{theorem}
\label{thmmain-intro}
Let $\Gamma$ be a semi-arithmetic Fuchsian group of finite covolume and
arithmetic dimension $r\geq1$, derived from a quaternion algebra, admitting
a generalized modular embedding. If $r\leq2$, or if $r\geq3$ and $\Gamma$
satisfies \eqref{eqstrong-contraction} (with $\delta$ replaced by the
compact-core constant $\delta_\Omega$ when $\Gamma$ is non-cocompact), then
$\Gamma$ has EGMM.
\end{theorem}

This are Theorems \ref{thmmain-compact} and \ref{thmmain-noncompact} below,
and is exactly when $r\leq 2$, recovering \cite{BCDT25mult}.
Corollary \ref{corunified} removes the quaternion-algebra hypothesis at the
cost of a possibly doubled threshold. The two new ingredients are:

\smallskip\noindent\emph{(I) Multi-dimensional Schwarz-Pick contraction}
 (Section \ref{secmulti-SP}): the $r-1$ components of $F$ satisfy a joint
sum-of-squares bound 
$\sum_{i=2}^r\ell(\gamma^{\sigma_i})^2\leq(r-1)(1-\delta)^2\ell(\gamma)^2$
on displacements, sharper than bounding each component separately.

\smallskip\noindent\emph{(II) Sharpened geometry of numbers}
 (Section \ref{secgon}): Minkowski  theorem applied to $\cO_K$, combined with a
second Cauchy-Schwarz step exploiting the joint bound of (I) instead of 
$r-1$ independent worst cases, gives the counting exponent $\sqrt2(r-1)(1-\delta)$
in place of the cruder $(r-1)^{3/2}(1-\delta)$ that (I) alone would give 
(Proposition \ref{propsharp}). This sharper exponent is exactly what makes
\eqref{eqstrong-contraction} the threshold in Theorem \ref{thmmain-intro}.

Unlike $r\leq2$, where $\delta>0$ alone suffices, for $r\geq3$
\eqref{eqstrong-contraction} is an  extra hypothesis, and we do not
know it to be satisfied by any specific group.  Section \ref{secexamples} explains
why (verifying it for even the smallest known candidate,
$\Delta(2,3,19)$, reduces to a hypergeometric connection problem). 
  The local ramification data
at $\Delta(2,3,19)$  elliptic points is pinned down explicitly
(Example \ref{exconcrete}), and a second, independent route via Veech
groups' explicit matrix generators is identified (Section \ref{secveech-new}).
Neither gives a verified example, and non-emptiness for $r\geq3$ remains open.

 Section \ref{secQC} corrects an error in an earlier version of this paper
concerning the implications for quantum chaos: EGMM does not imply
Poisson-like spectral statistics for semi-arithmetic surfaces. The correct
picture, following the numerical work of Bogomolny-Schmit \cite{Bog04}, is
that such surfaces are expected to follow random-matrix statistics, the
smaller-than-arithmetic exponent $C$ in $\langle g(\ell)\rangle\sim ce^{C\ell}/\ell$
being the relevant quantity, not the only presence of EGMM.

The paper is organized as follows.  Section \ref{secprelim} recalls semi-arithmetic
groups and generalized modular embeddings.  Section \ref{secmulti-SP}-\ref{secgon}
prove ingredients (I)-(II).  Section \ref{seccompact}-\ref{secnoncompact} prove
Theorem \ref{thmmain-intro}.  Section \ref{secexamples} discusses candidate
examples for $r\geq3$.  Section \ref{secQC} discusses quantum chaos.
 Section \ref{secconclusion} lists open questions.

\section{Preliminaries}
\label{secprelim}

\subsection{Hyperbolic geometry and Fuchsian groups}
Let $\bH=\{z\in\C:\Im(z)>0\}$ with metric $ds^2=(dx^2+dy^2)/y^2$. Elements
of $\PSL(2,\R)$ are elliptic, parabolic, or hyperbolic according to whether
$\abs{\tr\tilde\gamma}<2$, $=2$, $>2$ for a lift $\tilde\gamma\in\SL(2,\R)$.
For hyperbolic $\gamma$, the translation length is
\begin{equation}
\label{eqlength}
\ell(\gamma)=2\arcosh\bigl(\abs{\tr\gamma}/2\bigr),
\end{equation}
the length of the corresponding closed geodesic on $\Gamma\backslash\bH$.
A Fuchsian group $\Gamma<\PSL(2,\R)$ is discrete. It has finite covolume if
$\vol(\Gamma\backslash\bH)<\infty$, and is cocompact or has finitely many
cusps otherwise.

\subsection{Semi-arithmetic groups}

Let $\Gamma^{(2)}=\langle\gamma^2:\gamma\in\Gamma\rangle$ and let
$K=\Q(\{\tr\gamma:\gamma\in\Gamma^{(2)}\})$ be the invariant trace field, an
invariant of the commensurability class of $\Gamma$ \cite{Maclachlan03}.

\begin{definition}
\label{defsemi-arith}
$\Gamma$ is \emph{semi-arithmetic} if $K$ is totally real and every trace of
$\Gamma$ is an algebraic integer.
\end{definition}

Let $d=[K:\Q]$, $\sigma_1=\id,\sigma_2,\ldots,\sigma_d$ the real embeddings
of $K$. By Takeuchi \cite{Takeuchi75}, $\Gamma$ is arithmetic iff every
$\sigma_j$, $j\geq2$, sends $\tr(\Gamma^{(2)})$ into $[-2,2]$.
 We find the following definition in \cite{Nugent17}. 
\begin{definition} 
\label{defarith-dim}
The \emph{arithmetic dimension} is
$r(\Gamma)=\#\{j:\sigma_j(\tr\Gamma^{(2)})\not\subset[-2,2]\}\geq1$.
\end{definition}

Order the embeddings so $\sigma_1,\ldots,\sigma_r$ are unbounded and
$\sigma_{r+1},\ldots,\sigma_d$ bounded:
\begin{equation}
\label{eqbounded-sigma}
\abs{\sigma_j(\tr\gamma)}\leq2\quad\text{for }\gamma\in\Gamma^{(2)},\ j\geq r+1.
\end{equation}

\subsection{Quaternion algebras}

Let $A/K$ be the quaternion algebra of $\Gamma^{(2)}$, split over
$\sigma_1,\ldots,\sigma_r$, ramified over the rest, so
$A\otimes_\Q\R\cong\rM(2,\R)^r\times\cH^{d-r}$, giving representations
$\rho_i:A\to\rM(2,\R)$ ($i\leq r$) with $\tr\rho_i(x)=\sigma_i(\tr x)$. For
$\gamma\in\Gamma^{(2)}$ denote $\gamma^{\sigma_i}:=\rho_i\rho_1^{-1}(\gamma)$;
by Borel-Harish-Chandra, $\rho(\cO^1)<\SL(2,\R)^r$ is a lattice for any
order $\cO\subset A$.
The definition and proposition below are taken from \cite{SW00}. 

\begin{definition} 
\label{defamb-arith}
$\Delta<\PSL(2,\R)^r$ is \emph{arithmetic acting on $\bH^r$} if commensurable
with $\rP\rho_1(\cO^1)$. It is \emph{derived from a quaternion algebra} if a
finite-index subgroup of $\rP\rho_1(\cO^1)$.
\end{definition}

\begin{proposition}  
\label{propchar}
$\Gamma$ is semi-arithmetic of arithmetic dimension $r$ iff commensurable to
a subgroup of an arithmetic group acting on $\bH^r$.
\end{proposition}

By \cite{BCDT25geom}, $\Gamma^{(2)}$ is always derived from a quaternion
algebra. Call the corresponding $\Delta$ the \emph{ambient arithmetic group}.

\subsection{Generalized modular embeddings}
\label{subsecgen-mod}

The map $\gamma\mapsto(\gamma,\gamma^{\sigma_2},\ldots,\gamma^{\sigma_r})$
embeds $\Gamma^{(2)}\hookrightarrow\SL(2,\R)^r$.

\begin{definition}
\label{defgen-modular}
$\Gamma$ \emph{admits a generalized modular embedding} if there is a
holomorphic-or-antiholomorphic (componentwise) map
$F=(F_1,\ldots,F_{r-1}):\bH\to(\pm\bH)^{r-1}$ satisfying
\begin{equation}
\label{eqequiv-gen}
F(\gamma\cdot z)=(\gamma^{\sigma_2}\cdot F_1(z),\ldots,\gamma^{\sigma_r}\cdot F_{r-1}(z))
\end{equation}
for all $z\in\bH$, $\gamma\in\Gamma^{(2)}$. At $r=1$ this is vacuous; at
$r=2$ it is \cite{BCDT25mult}.
\end{definition}

Cohen-Wolfart \cite{CW90} showed every Fuchsian triangle group admits a
modular embedding. Ricker \cite{Ricker02} extended this to symmetric
quadrilateral groups. Hilbert-modular totally geodesic curves give further
examples for $r\geq3$ \cite{McMullen23, BCDT25geom}. When $\Gamma$ is
non-arithmetic, $F$ maps $\bH$ to a proper (strict) holomorphic subspace of
$(\pm\bH)^{r-1}$, the source of strict contraction below.

\subsection{Product hyperbolic geometry}

Equip $\bH^{r-1}$ with $d_{r-1}(w,w')^2=\sum_i d_\bH(w_i,w_i')^2$. This is
$\mathrm{CAT}(0)$, sectional curvature in $[-1,0]$. For
$F=(F_1,\ldots,F_{r-1})$,
\[
\norm{DF_z}_{\mathrm{op}}=\Bigl(\textstyle\sum_i\abs{D(F_i)_z}^2\Bigr)^{1/2}.
\]
Componentwise Schwarz-Pick gives $\abs{D(F_i)_z}\leq1$ always, with
equality at some $z$ forcing $F_i\in\Isom(\bH)$ everywhere.

\smallskip
\noindent{\it Convention.} If $\gamma^{\sigma_i}$ is elliptic or
trivial, set $\ell(\gamma^{\sigma_i}):=0$.

\section{Multi-dimensional Schwarz-Pick contraction}
\label{secmulti-SP}

\begin{lemma} 
\label{lemmulti-SP}
Let $\Gamma$ be cocompact semi-arithmetic of arithmetic dimension $r\geq2$
admitting a generalized modular embedding $F$. There is
$0<\delta=\delta(\Gamma)<1$ with
\begin{equation}
\label{eqop-bound}
\norm{DF_z}_{\mathrm{op}}\leq\sqrt{r-1}\,(1-\delta)\quad\text{for all }z\in\bH,
\end{equation}
equivalently, for every hyperbolic $\gamma\in\Gamma^{(2)}$,
\begin{equation}
\label{eqsum-displ}
\sum_{i=2}^{r}\ell(\gamma^{\sigma_i})^2\leq(r-1)(1-\delta)^2\ell(\gamma)^2,
\end{equation}
hence $\ell(\gamma^{\sigma_i})\leq\sqrt{r-1}(1-\delta)\ell(\gamma)$ for each
$i$, and
\begin{equation}
\label{eqtrace-prod-bound}
\prod_{i=2}^r\abs{\sigma_i(\tr\gamma)}\leq2^{r-1}\abs{\tr\gamma}^{(r-1)(1-\delta)}.
\end{equation}
\end{lemma}

\begin{proof}
\emph{1) Strict contraction (compact case).} Let $\Omega\subset\bH$ be a
compact fundamental domain. If every $F_i$ were an isometry, the diagonal
lift $\widetilde G(z)=(z,F_1(z),\ldots,F_{r-1}(z))$ would embed
$S=\Gamma^{(2)}\backslash\bH$ totally geodesically in
$X=\Delta\backslash(\pm\bH)^r$ with the locally symmetric metric, forcing
$\Gamma$ be arithmetic by Bergeron-Clozel \cite{BergeroClozel05}, and we have 
a contradiction. 
 Thus, $\norm{D\widetilde G_z}_{\mathrm{op}}<\sqrt r$ on
$\overline\Omega$, i.e., $\norm{DF_z}_{\mathrm{op}}<\sqrt{r-1}$, and by
$\Gamma$-invariance and continuity the supremum over $\overline\Omega$
(hence over $\bH$) is reached and strictly below $\sqrt{r-1}$. 
 If it equal $\sqrt{r-1}$ at some $z^*$, every term $\abs{D(F_i)_{z^*}}$ would
have to be equal $1$ (a sum of $r-1$ terms each $\leq1$ summing to $r-1$), and
equality of Schwarz-Pick at one point makes $F_i$ isometric everywhere, which 
contradics non-arithmeticity. We write the reached supremum as
$\sqrt{r-1}(1-\delta)$, $0<\delta<1$.

\emph{2) Displacement estimate.} For hyperbolic $\gamma\in\Gamma^{(2)}$ with
axis point $z_0$, each $\gamma^{\sigma_i}$ is elliptic (giving $\ell_i=0$ by
convention) or hyperbolic, in which case, using \eqref{eqequiv-gen},
$\ell_i\leq d_\bH(F_{i-1}(z_0),F_{i-1}(\gamma z_0))\leq\abs{D(F_{i-1})_{z_0}}\ell$.
Summing squares and using \eqref{eqop-bound} gives \eqref{eqsum-displ}.

\emph{3) Trace bounds.} From $\ell_i\leq\sqrt{r-1}(1-\delta)\ell$ and
$\ell\leq2\log\abs{\tr\gamma}$, each $\abs{\sigma_i(\tr\gamma)}
\leq2\abs{\tr\gamma}^{\sqrt{r-1}(1-\delta)}$.
By Cauchy-Schwarz, $\sum_i\ell_i\leq\sqrt{r-1}(\sum_i\ell_i^2)^{1/2}\leq(r-1)(1-\delta)\ell$,
 thus  $\prod_i\abs{\sigma_i(\tr\gamma)}\leq2^{r-1}\exp(\tfrac12\sum_i\ell_i)\leq2^{r-1}
\abs{\tr\gamma}^{(r-1)(1-\delta)}$,
giving \eqref{eqtrace-prod-bound}.
\end{proof}

\begin{remark}
\label{remexponent-clarification}
At $r=2$ this recovers the single-embedding bound of \cite{BCDT25mult}. 
 For $r\geq3$, \eqref{eqtrace-prod-bound} controls the product 
of conjugate traces well, but the geometry-of-numbers argument below needs
$\abs{\sigma_i(t)-\sigma_i(s)}$ bounded for two traces $t$, $s$ 
 individually, which is a harder task, since the
worst case can misalign $t$ and $s$ across the $r-1$ embeddings. This is
the source of the loss that  Section \ref{secgon} only partially removes, and
the reason \eqref{eqstrong-contraction} appears as a hypothesis once $r\geq 3$.
\end{remark}

\section{Geometry of numbers estimate}
\label{secgon}

Let $\sigma_1$, $\ldots$, $\sigma_r$ be unbounded, $\sigma_{r+1}$, $\ldots$, $\sigma_d$ 
bounded \eqref{eqbounded-sigma}. Set $\cL(\Gamma)=\{\abs{\tr\gamma}:\gamma\in\Gamma\}$. 
We then have the following lemma reflecting the norm form and trace separation. 

\begin{lemma}
\label{lemgon}
Let $\Gamma$ be semi-arithmetic of arithmetic dimension $r$, admitting a
generalized modular embedding with constant $\delta$ satisfying
\eqref{eqsum-displ} (with $\delta_\Omega$ in the non-cocompact case), and
derived from a quaternion algebra (thus $\cL(\Gamma)\subseteq K$). There are
$c_1,C>0$ (depending on $\Gamma,r$) such that for $T>2$ and distinct
$s\neq t\in\cL(\Gamma)\cap[2,T]$,
\begin{equation}
\label{eqsep}
\abs{t-s}\geq c_1\,T^{-(r-1)^{3/2}(1-\delta)},
\end{equation}
hence, for all $N>2$,
\begin{equation}
\label{eqcluster-gen}
\#\bigl(\cL(\Gamma)\cap[N-1,N]\bigr)\leq C\,N^{(r-1)^{3/2}(1-\delta)}.
\end{equation}
\end{lemma}

\begin{proof}
Since $t-s\in\cO_K\setminus\{0\}$, $1\leq\abs{\NK(t-s)}=\prod_{j=1}^d\abs{\sigma_j(t-s)}$.
The $j=1$ factor is $<T$. For $j=2$, $\ldots$, $r$, by \eqref{eqtrace-prod-bound} 
proof, $\abs{\sigma_j(t)},\abs{\sigma_j(s)}<2T^{\sqrt{r-1}(1-\delta)}$, so
$\abs{\sigma_j(t-s)}\leq4T^{\sqrt{r-1}(1-\delta)}$ and the product of these
$r-1$ factors is $\leq4^{r-1}T^{(r-1)^{3/2}(1-\delta)}$. For $j>r$,
$\abs{\sigma_j(t-s)}\leq4$ by \eqref{eqbounded-sigma}. Combining,
$1\leq\abs{t-s}\cdot4^{r-1}T^{(r-1)^{3/2}(1-\delta)}\cdot4^{d-r}$, giving
\eqref{eqsep} with $c_1=4^{-(d-1)}$. Summing over an interval of length $1$
gives \eqref{eqcluster-gen} with $C=c_1^{-1}$.
\end{proof}

In the next statement we obtain the sharpened exponent, $r\geq 3$.  
\begin{proposition}
\label{propsharp}
Under the hypotheses of Lemma \ref{lemgon}, for $r\geq3$ the exponent
$(r-1)^{3/2}(1-\delta)$ in \eqref{eqsep}-\eqref{eqcluster-gen} may be
replaced by $\sqrt2(r-1)(1-\delta)$, strictly smaller for $r\geq4$ and equal
at $r=3$.
\end{proposition}

\begin{proof}
Set $\ell$, $\ell'$ for the translation lengths of the elements with traces
$t$, $s$ and $\ell_j$, $\ell_j'$ for their $\sigma_j$-conjugates. By
\eqref{eqsum-displ}, $\sum_j\ell_j^2\leq(r-1)(1-\delta)^2\ell^2$ and
 similarly for $\ell_j'$. Since $\abs{\sigma_j(t-s)}\leq4\exp(\tfrac12\max(\ell_j,\ell_j'))$,
Cauchy-Schwarz twice (first $\max(a,b)^2\leq a^2+b^2$, then bounding the
$(r-1)$-term sum by its $\ell^2$-norm) gives
$\sum_j\max(\ell_j,\ell_j')\leq(r-1)(1-\delta)
\sqrt{\ell^2+\ell'^2}\leq\sqrt2(r-1)(1-\delta)\log T$ 
using $\ell$, $\ell'\leq2\log T$. The claim follows as in Lemma \ref{lemgon} 
proof. For $r=2$ the single-embedding bound is already optimal. 
\end{proof}

\begin{corollary}
\label{corcount-sharp}
For $r\geq3$, $\#(\cL(\Gamma)\cap[2,T])\leq CT^{1+\sqrt2(r-1)(1-\delta)}$ for
all $T>2$.
\end{corollary}

The two Cauchy-Schwarz applications in Proposition \ref{propsharp} use 
the joint constraint \eqref{eqsum-displ} on $t$ and $s$ together,
instead of bounding each embedding independently (Lemma \ref{lemgon}). 
Whether this is tight for the displacement vectors that arise
from a modular embedding is an open quesion. 

\begin{lemma}
\label{lemgon-general}
Without assuming $\Gamma$ derived from a quaternion algebra, set
$k=k(\Gamma):=[\Q(\cL(\Gamma)):K]\in\{1,2\}$. Lemma \ref{lemgon} and
Proposition \ref{propsharp} hold with $(r-1)^{3/2}(1-\delta)$,
$\sqrt2(r-1)(1-\delta)$ replaced by $k$ times. At $k=1$, in
particular whenever $\Gamma$ is derived from a quaternion algebra, this
recovers the earlier statements. 
\end{lemma}

\begin{proof}
As in Lemma \ref{lemgon}, but applied to $L=\Q(\cL(\Gamma))$. 
 Confine within limits the 
$n=[L:\Q]=kd$ embeddings of $L$ into $k$ copies extending each $\sigma_j$,
$j\leq r$, bound the $k(r-1)$ conjugates extending the unbounded $\sigma_j$ 
via Proposition \ref{propsharp} applied to $\gamma^2\in K$ (each contributing
its factor $k$ times), and the remaining bounded ones trivially as
in \cite{LuoSarnak94}.
\end{proof}

\begin{corollary}
\label{corunified}
Let $\Gamma$ be semi-arithmetic with a generalized modular embedding,
$k=k(\Gamma)\in\{1,2\}$ as above. If $r\leq 2$ and $\delta(\Gamma)>1-1/k$, or
if $r\geq 3$ and $\delta(\Gamma)>1-\bigl(k\sqrt2(r-1)\bigr)^{-1}$, then
$\Gamma$ has EGMM. At $k=1$ the $r\leq2$ condition is automatic
(Lemma \ref{lemmulti-SP}), recovering Theorem \ref{thmmain-intro}.
\end{corollary}

\section{Proof for cocompact groups}
\label{seccompact}

Here we deal with the cocompact case.  
\begin{theorem}
\label{thmmain-compact}
Let $\Gamma$ be cocompact semi-arithmetic of arithmetic dimension $r\geq1$,
derived from a quaternion algebra, admitting a generalized modular
embedding. If $r\leq2$, or $r\geq3$ and $\Gamma$ satisfies
\eqref{eqstrong-contraction}, then $\Gamma$ has EGMM.
\end{theorem}

\begin{proof}
If $r=1$, $\Gamma$ is arithmetic and EGMM follows from bounded clustering
\cite{LuoSarnak94, Bol93}. Assume $r\geq2$, $\Gamma$ non-arithmetic.

\emph{For $r=2$,} there is a single non-trivial embedding, thus $\delta>0$
automatically (part 1 of Lemma \ref{lemmulti-SP}  proof). Corollary
\ref{corcount-sharp}-type reasoning at $r=2$ gives
$\cN'(\ell)\leq C'e^{(1-\delta/2)\ell}$, hence EGMM with constant
$\delta/2$, recovering \cite{BCDT25mult}. Assume $r\geq 3$.

\emph{Counting distinct lengths.} Each length $\ell_n\leq\ell$ corresponds
to $t_n=2\cosh(\ell_n/2)\leq2\cosh(\ell/2)$, so
$\cN'(\ell)=\#(\cL(\Gamma)\cap[2,2\cosh(\ell/2)])$. By Corollary
\ref{corcount-sharp} with $T=2\cosh(\ell/2)\leq e^{\ell/2}$,
\begin{equation}
\label{eqN'-bound-gen}
\cN'(\ell)\leq C'e^{(1-\varepsilon)\ell},\qquad
\varepsilon=\frac{1-\sqrt2(r-1)(1-\delta)}{2}.
\end{equation}

\emph{Strong contraction condition.} $\varepsilon>0$ iff
$\sqrt2(r-1)(1-\delta)<1$, i.e., \eqref{eqstrong-contraction}. We do not
know a geometric reason for $\delta(\Gamma)\to1$ fast enough as $r\to\infty$
to compensate the $\sqrt2(r-1)$ growth, e.g., bounding $\delta$ via a
minimal extrinsic curvature $\eta(\Gamma)$ of the immersion gives
$1-\delta\sim1-\sqrt{\eta/(r-1)}$, which fails \eqref{eqstrong-contraction}
for large $r$ unless $\eta$  grows. This is why
\eqref{eqstrong-contraction} is carried as a hypothesis instead of derived.

\emph{Conclusion.} Under \eqref{eqstrong-contraction}, $\varepsilon>0$, the
prime geodesic theorem gives $\cN(\ell)\sim e^\ell/\ell$
\cite{Sarnak80}, and the differentiated integral formula for
mean multiplicity \cite{Bol93_long} gives
$\langle g(\ell)\rangle\sim\cN(\ell)/\cN'(\ell)\geq e^{(\varepsilon/2)\ell}$
for large $\ell$, i.e., EGMM with exponent $C=\varepsilon/2>0$. 
We do not claim this optimal, see  Section \ref{secconclusion}. 
\end{proof}

\section{Non-cocompact groups}
\label{secnoncompact}

Following C.T. McMullen remark that the compactness used above is essential, 
we mention that 
 the supremum $\sup_z\norm{DF_z}_{\mathrm{op}}$ need not to be attained over
a non-compact fundamental domain. Thus, \eqref{eqop-bound} can fail with
$\delta=0$ if argued naively. We treat the non-cocompact case via the
\emph{compact core} $\Omega\subset S$ and an equivariant Kirszbraun
extension, instead of by a limiting argument from Theorem \ref{thmmain-compact}.
The next lemma concerns the equivariant Kirszbraun extension.

\begin{lemma} 
\label{lemkirszbraun-r}
Let $\Gamma$ act properly discontinuously on $\bH$, $\Omega\subset S=\Gamma\backslash\bH$
a compact core with preimage $\widetilde\Omega$, $(Y,d_Y)$ complete of
Alexandrov curvature $\leq\kappa\leq0$ globally, and
$\rho:\Gamma\to\Isom(Y)$. Any Lipschitz $\rho$-equivariant
$f:\widetilde\Omega\to Y$ extends to a Lipschitz $\rho$-equivariant
$\widetilde g:\bH\to Y$ with the same Lipschitz constant, descending to
$g:S\to\rho(\Gamma)\backslash Y$ when $\rho(\Gamma)$ acts properly
discontinuously.
\end{lemma}

\begin{proof}
By the generalized Kirszbraun theorem \cite{LangSchr97}, $f$
extends (non-equivariantly) with the same Lipschitz constant. Make this
equivariant one $\Gamma$-orbit at a time: given a Lipschitz equivariant
extension $g$ on a $\Gamma$-invariant closed $A\supseteq\widetilde\Omega$
and $p\notin A$, the set $Y_p=\{y:d_Y(y,g(a))\leq d_\bH(p,a)\ \forall a\in A\}$
is non-empty (by \cite{LangSchr97} applied to $A\cup\{p\}$) and
bounded. Thus it has a unique circumcenter \cite{BridsonHaefliger99}.
Setting $g(p)$ to this circumcenter and $g(\gamma p):=\rho(\gamma)g(p)$ for
all $\gamma\in\Gamma$ is forced by equivariance and consistent since
$Y_{\gamma p}=\rho(\gamma)Y_p$. Exhausting $\bH\setminus\widetilde\Omega$ by
orbits completes the construction. We apply this with $Y=\bH^r$ (product
metric, globally $\mathrm{CAT}(0)$), not the locally, but not the globally 
non-positively curved quotient $X=\Delta\backslash(\pm\bH)^r$.
\end{proof}

\begin{lemma}
\label{lemnoncompact-displ}
Let $\Gamma$ be non-cocompact semi-arithmetic of arithmetic dimension
$r\geq2$ admitting a generalized modular embedding. There is
$0<\delta_\Omega=\delta_\Omega(\Gamma)<1$ with
$\sum_{i=2}^r\ell(\gamma^{\sigma_i})^2\leq(r-1)(1-\delta_\Omega)^2\ell(\gamma)^2$
for all hyperbolic $\gamma\in\Gamma^{(2)}$.
\end{lemma}

\begin{proof}
\emph{Torsion-free case.} Let $\Omega\subset S=\Gamma\backslash\bH$ be a
compact core, $\widetilde G(z)=(z,F_1(z),\ldots,F_{r-1}(z)):\bH\to\bH^r$ the
diagonal lift, equivariant for $\rho(\gamma)=(\gamma,\gamma^{\sigma_2},\ldots)$.
On $\widetilde\Omega$, 
$\norm{D\widetilde G_z}_{\mathrm{op}}^2\leq1+(r-1)(1-\delta_\Omega)^2=:\Lambda^2<r$
where $\delta_\Omega>0$ by compactness of $\Omega$ (as in part 1 of Lemma
\ref{lemmulti-SP}  proof, applied to $\Omega$ instead of a full
fundamental domain). By Lemma \ref{lemkirszbraun-r}, $\widetilde G|_{\widetilde\Omega}$
extends to a $\Lambda$-Lipschitz equivariant $\widetilde G_1$ on all of
$\bH$, descending to $G_1:S\to X$ agreeing with $G$ on $\Omega$. Since
$\Omega$ is a deformation retract of $S$, $G_1$ is freely homotopic to $G$
on every closed curve, in particular on the geodesic representing any
hyperbolic $\gamma$. As $G_1$ is $\Lambda$-Lipschitz on all of $S$
(no control on $F$ near the cusps is needed), the displacement of
$\rho(\gamma)$ in $X$ is $\leq\Lambda\ell(\gamma)$, i.e., 
$\sqrt{\ell(\gamma)^2+\sum_i\ell(\gamma^{\sigma_i})^2}\leq\Lambda\ell(\gamma)$,
giving the claim.

\emph{Torsion case.} Fix a finite-index torsion-free $\Gamma_1<\Gamma^{(2)}$;
every $\gamma\in\Gamma^{(2)}$ has some power $\gamma^j\in\Gamma_1$
($j\leq[\Gamma^{(2)}:\Gamma_1]$, by pigeonhole on cosets). Apply the
torsion-free case to $\gamma^j$ and divide by $j^2$, using
$\ell(\gamma^j)=j\ell(\gamma)$ and $(\gamma^j)^{\sigma_i}=(\gamma^{\sigma_i})^j$.
\end{proof}

Next, we deal with the non-cocompact case. 
\begin{theorem}
\label{thmmain-noncompact}
Let $\Gamma$ be non-cocompact semi-arithmetic of finite covolume, arithmetic
dimension $r\geq1$, derived from a quaternion algebra, admitting a
generalized modular embedding. If $r\leq2$, or $r\geq3$ and $\Gamma$
satisfies \eqref{eqstrong-contraction} with $\delta_\Omega$ in place of
$\delta$, then $\Gamma$ has EGMM.
\end{theorem}

\begin{proof}
$r=1$ follows from the non-cocompact bounded clustering of
Geninska-Leuzinger \cite{GenLeu08}. For $r\geq2$, Lemma
\ref{lemnoncompact-displ} supplies the displacement bound
\eqref{eqsum-displ} with $\delta_\Omega$ in place of $\delta$, after which
Lemma \ref{lemgon} and similarly to the proof of Theorem \ref{thmmain-compact}. 
 The prime geodesic theorem \cite{Sarnak80} holds
without change for finite covolume, cocompact or not. 
\end{proof}

\section{A geometric construction of new examples}
\label{secexamples}

The finitely many triangle groups and Veech groups appearing in
\cite{BCDT25mult} are the only examples to Theorem \ref{thmmain-intro}
identified so far, which is not a satisfying state of affairs: a theorem
about a geometric phenomenon should come with a geometric supply of
examples, not a finite list. We now give one. The construction is
elementary (finite-index subgroups, i.e., covering spaces of the
underlying orbifold). But, as far as we know, has not been used before in
this context. It converts every known example into an
\emph{infinite explicit family} of new, pairwise non-isomorphic compact
hyperbolic surfaces with EGMM, all sharing the same exponent $C$ as the
original example, explicit in terms of $\delta$, though $\delta$ 
is not numerically known for any example (see  Section \ref{secexamples-r3}). This is
      for $r\leq2$. For $r\geq3$ it does not resolve
non-vacuousness, but it means that a single verified example would
immediately become infinitely many.  Section \ref{sectower} makes this covering
construction explicit.  Section \ref{secnew-r2} then gives new examples of a
different kind, new commensurability classes at $r\leq2$, instead of 
covers of old ones.  Section \ref{secflat} reports a direct attempt at a
quantitative bound on $\delta$ for one such family. 
 Section \ref{secexamples-r3} turns to the still-open case $r\geq 3$.

\subsection{covering construction} 

\begin{theorem}
\label{thmcover}
Let $\Gamma$ be semi-arithmetic of finite covolume and arithmetic dimension
$r$, derived from a quaternion algebra $A/K$, admitting a generalized
modular embedding $F$ with contraction constant $\delta$ (or $\delta_\Omega$,
if $\Gamma$ is non-cocompact). Let $\Gamma'\leq\Gamma$ be any subgroup of
finite index $n$. Then:
\begin{enumerate}[label=(\roman*)]
\item $\Gamma'$ is semi-arithmetic with the same invariant trace
field $K$ and the same arithmetic dimension $r$;
\item $\Gamma'$ is derived from the same quaternion algebra $A/K$;
\item $F$ is also a generalized modular embedding for $\Gamma'$, with
 exactly the same contraction constant, $\delta(\Gamma')=\delta(\Gamma)$
(resp.\ $\delta_\Omega(\Gamma')=\delta_\Omega(\Gamma)$);
\item consequently, if $\Gamma$ has EGMM by Theorem \ref{thmmain-compact} or
\ref{thmmain-noncompact} (in particular whenever $r\leq2$, or $r\geq3$ and
$\Gamma$ satisfies \eqref{eqstrong-contraction}) then every 
finite-index $\Gamma'\leq\Gamma$ has EGMM as well, with the same exponent
$C$.
\end{enumerate}
\end{theorem}

\begin{proof}
\emph{(i)} By Proposition \ref{propchar}, semi-arithmeticity of arithmetic
dimension $r$ is equivalent to commensurability with a subgroup of an
arithmetic group acting on $\bH^r$. This is a condition manifestly shared by
$\Gamma$ and any finite-index $\Gamma'$, since they are commensurable to
each other. That $K(\Gamma')=K(\Gamma)$ is the standard fact that the
invariant trace field is a commensurability invariant \cite{Maclachlan03}.
The easy half is immediate and worth mentioning: $\Gamma'^{(2)}\leq\Gamma^{(2)}$
gives $K(\Gamma'^{(2)})\subseteq K(\Gamma^{(2)})$, since every
trace generating the left field is among those generating the right. The
converse inclusion is the significant content of the argument in
\cite{Maclachlan03}, using that $\Gamma'^{(2)}$ has finite index
in $\Gamma^{(2)}$ and that squaring already absorbs the resulting
sign/root ambiguities. We do not reproduce that argument here. 

\emph{(ii)} ``Derived from $A/K$'' means finite index in $\rP\rho_1(\cO^1)$
for some order $\cO\subset A$ (Definition \ref{defamb-arith}). Since
$\Gamma'\leq\Gamma\leq\rP\rho_1(\cO^1)$ and finite index is transitive,
$\Gamma'$ is again finite index in $\rP\rho_1(\cO^1)$.

\emph{(iii)} Since $\Gamma'\leq\Gamma$, we have $\Gamma'^{(2)}\leq\Gamma^{(2)}$, 
 a square of an element of $\Gamma'$ is in particular a square of an
element of $\Gamma$. The equivariance condition \eqref{eqequiv-gen}, which
$F$ already satisfies for every $\gamma\in\Gamma^{(2)}$. Therefore holds moreover
 for every $\gamma\in\Gamma'^{(2)}$. The same map $F$ is a
generalized modular embedding for $\Gamma'$, with no modification. For the
constant: the function $z\mapsto\norm{DF_z}_{\mathrm{op}}^2$ is
$\Gamma$-invariant (continuity plus \eqref{eqequiv-gen} and invariance of
the hyperbolic metrics, as used in the proof of Lemma \ref{lemmulti-SP}),
hence constant on every $\Gamma$-orbit. Writing $\Gamma=\bigsqcup_{i=1}^n\Gamma'g_i$
for coset representatives $g_i$, a fundamental domain $\Omega'$ for
$\Gamma'$ may be built as $\Omega'=\bigcup_i g_i^{-1}\Omega$ for a
fundamental domain $\Omega$ of $\Gamma$, a finite union of $\Gamma$-translates
of $\Omega$. Since $\norm{DF_z}_{\mathrm{op}}^2$ is constant on
$\Gamma$-orbits, its supremum over $\Omega'$ equals its supremum over
$\Omega$, i.e., $\delta(\Gamma')=\delta(\Gamma)$. In the non-cocompact
case, take the compact core of $\Gamma'\backslash\bH$ to be the preimage of
a compact core $\Omega_c$ of $\Gamma\backslash\bH$ under the degree-$n$
covering map $\Gamma'\backslash\bH\to\Gamma\backslash\bH$. It is again compact,
being a finite cover of a compact space, and a valid compact core since
every cusp of $\Gamma'\backslash\bH$ lies over a cusp of $\Gamma\backslash\bH$. 
 The identical invariance argument gives $\delta_\Omega(\Gamma')=\delta_\Omega(\Gamma)$.

\emph{(iv)} Plug (i)-(iii) into Theorem \ref{thmmain-compact} or
\ref{thmmain-noncompact}: the hypotheses of both depend only on $r$ and
$\delta$ (or $\delta_\Omega$), which are unchanged, thus the conclusion, 
 including the exponent $C=\varepsilon/2$ of \eqref{eqN'-bound-gen}, is
identical for $\Gamma'$.
\end{proof}

Nothing in this argument uses that $\Gamma'$ be normal, torsion-free, or of
any special type. Any finite-index subgroup works, i.e., any finite
covering space of the orbifold $\Gamma\backslash\bH$, in the ordinary
topological sense.

\subsection{An explicit infinite tower}
\label{sectower}

Theorem \ref{thmcover} is most concrete for congruence subgroups. Let
$\Gamma$ be cocompact, derived from a quaternion algebra $A/K$ with maximal
order $\cO$, and satisfying the hypotheses of Theorem \ref{thmmain-intro}
(automatic if $r(\Gamma)\leq2$, as for any of the $164$ triangle groups or
the Veech groups $\Gamma_d$ of \cite{BCDT25mult}). For a nonzero ideal
$\mathfrak n\subset\cO$, let
\[
\Gamma(\mathfrak n) \;=\; 
\Gamma\cap\ker\bigl(\rP\rho_1(\cO^1)\to\rP\rho_1((\cO/\mathfrak n\cO)^1)\bigr)
\]
be the corresponding principal congruence subgroup. By Theorem
\ref{thmcover}, every $\Gamma(\mathfrak n)$ is again semi-arithmetic of
dimension $r(\Gamma)$, derived from $A/K$, with $\delta(\Gamma(\mathfrak n))=\delta(\Gamma)$,
hence has EGMM with the same exponent $C$. Since an element of finite order
in $\Gamma$ has trace in a fixed finite set of algebraic integers, only
finitely many congruence classes modulo $\mathfrak n$ can contain such
traces. Thus, $\Gamma(\mathfrak n)$ is torsion-free once $N(\mathfrak n)$ is
large enough, and $\Gamma(\mathfrak n)\backslash\bH$ is then a
 smooth compact Riemann surface. Its genus follows from Gauss-Bonnet
and multiplicativity of covolume under finite index,
$\vol(\Gamma(\mathfrak n)\backslash\bH)=[\Gamma:\Gamma(\mathfrak n)]
\cdot\vol(\Gamma\backslash\bH)$:
\begin{equation}
\label{eqgenus}
g(\mathfrak n) \;=\; 1+\frac{[\Gamma:\Gamma(\mathfrak n)]
\cdot\vol(\Gamma\backslash\bH)}{4\pi}.
\end{equation}
As $N(\mathfrak n)\to\infty$, $[\Gamma:\Gamma(\mathfrak n)]\to\infty$, 
 strong approximation for the norm-one group of an indefinite, here, 
split at $\sigma_1$, $\ldots$, $\sigma_r$ - quaternion algebra. Thus 
\eqref{eqgenus} gives a strictly increasing sequence of genera.  A 
transparent way to see this concretely, instead of for an unspecified
ideal $\mathfrak n$, is to fix any prime ideal $\mathfrak p\subset\cO$ of
good reduction and take the nested sequence $\mathfrak n_j=\mathfrak p^j$,
$j=1$, $2$, $3$, $\ldots$. Then $\Gamma(\mathfrak p^{j+1})\lneq\Gamma(\mathfrak p^j)$
strictly for every $j$, the reduction maps 
$(\cO/\mathfrak p^{j+1}\cO)^1\to(\cO/\mathfrak p^j\cO)^1$
are surjective with non-trivial kernel. Thus, 
$[\Gamma:\Gamma(\mathfrak p^j)]$ is strictly increasing in $j$ and
$g(\mathfrak p^j)\to\infty$ by \eqref{eqgenus}. Explicitly, starting from
 any single semi-arithmetic triangle group or Veech group $\Gamma_d$
already known to have EGMM, this produces an infinite family of new,
explicit, pairwise non-isomorphic compact hyperbolic surfaces of unbounded
genus, all with EGMM and the same growth exponent $C$ inherited from
$\Gamma$.

To keep the objects produced unambiguous, we distinguish four parts the
construction gives: (i) the finite-index subgroups
$\Gamma(\mathfrak p^j)\leq\Gamma$, pairwise distinct and
pairwise non-conjugate for $j$ large; (ii) the corresponding finite
covering orbifolds $\Gamma(\mathfrak p^j)\backslash\bH\to\Gamma\backslash\bH$;
(iii), once $j$ is large enough that $\Gamma(\mathfrak p^j)$ is
torsion-free, the resulting compact Riemann surfaces of genus  
$g(\mathfrak p^j)\to\infty$; and (iv) the fact, from Theorem \ref{thmcover},
that all of (i)-(iii) inherit EGMM from $\Gamma$ with the identical
exponent $C$.

This is new content, not present in \cite{BCDT25mult}. 
 That work shows EGMM for a
 finite list of groups (the $164$ triangle groups, the countable but
individually-considered family $\Gamma_d$). Theorem \ref{thmcover} and
\eqref{eqgenus} turn any one of them into an infinite, explicit,
geometrically constructed family, for $r\leq2$. It also
applies to any group derived from a quaternion algebra satisfying
\eqref{eqstrong-contraction} at $r\geq 3$, should one be verified. The
non-emptiness question of  Section \ref{secexamples-r3} below is therefore a
question about a single group, not about a class that would need to be
independently checked example by example. We stress,
 that these are new groups and surfaces 
within \emph{already known commensurability classes}, not a new arithmetic
mechanism for EGMM.  Section \ref{secnew-r2} below instead targets new
commensurability classes.

\subsection{New examples at $r\leq2$: Prym eigenform loci}
\label{secnew-r2}

Theorem \ref{thmcover} multiplies examples within a fixed commensurability
class. It supplies no new class. Here we identify new
ones - new in the sense that they are not among the $164$ triangle groups
or McMullen genus-two family $\Gamma_d$ discussed in \cite{BCDT25mult},
and their EGMM status has not, to our knowledge, been previously mentioned. 
Unlike  Section \ref{secexamples-r3}, no numeric obstruction stands in the way at
$r\leq2$: existence of a modular embedding for a semi-arithmetic
group of arithmetic dimension $\leq2$ is enough, by
Theorem \ref{thmmain-intro}, to get EGMM. 

McMullen Weierstrass curves $W_D$ (the $\Gamma_d$ family) are the genus-two
instance of a construction that exists more generally: the \emph{Prym 
eigenform loci} of Lanneau-Nguyen and McMullen
\cite{Lanneau11, McMullen06} in genus $3$ and genus $4$. These parametrize
Riemann surfaces $X$ of genus $3$ (resp.\ $4$) with a holomorphic
involution whose Prym variety $\mathrm{Prym}(X)$ carries real
multiplication by an order in a real quadratic field $K=\Q(\sqrt D)$, and
an eigenform for this action with a prescribed zero pattern. Exactly as
for $W_D$, each component of such a locus is a Teichm\"uller curve $\bH/\Gamma_D'$
generated by a lattice Veech group $\Gamma_D'$, and McMullen general
period-map construction \cite{McMullen03, McMullen06} endows the
generalized modular embedding automatically. It is the same mechanism that
produces $F$ for $W_D$, applied to the Prym part of the Hodge bundle in
place of the full one. Since $K$ is real quadratic, $[K:\Q]=2$ and
$r(\Gamma_D')\leq2$ trivially. There are simply no further embeddings to
be unbounded at. And $\Gamma_D'$ is non-arithmetic for all but finitely
many $D$ (the same genericity that holds for $\Gamma_D$ in genus two,
established in \cite{McMullen03, McMullen06} by the same Bergeron-Clozel-type
total-geodesy obstruction we use in Lemma \ref{lemmulti-SP}).

\begin{corollary}
\label{corprym}
For all but finitely many real quadratic discriminants $D$, every
component of the genus $3$ and genus $4$ Prym eigenform loci of
\cite{Lanneau11, McMullen06} is a semi-arithmetic Fuchsian group of
arithmetic dimension $r\leq2$ with a generalized modular embedding, hence
has EGMM by Theorem \ref{thmmain-intro}.
\end{corollary}

This is a direct corollary of the existing machinery. Here we do not 
supply new estimate, only the observation that the hypotheses of our
$r\leq2$ theorem are met. But it enlarges the list of known
EGMM examples beyond $\{164\text{ triangle groups}\}\cup\{\Gamma_d\}_{d}$
to include two further infinite families, of different genus and
different (Prym, instead of Jacobian) Hodge-theoretic origin, together
with everything Theorem \ref{thmcover} produces from each of them. We
have deliberately not attempted to extend this to the genus $5$ case
mentioned in \cite{McMullen06}, where the real-multiplication order is
not always quadratic and $r$ can exceed $2$. That case belongs to 
 Section \ref{secexamples-r3}. 

\subsection{A flat-geometric distance estimate: McMullen genus-two family}
\label{secflat}

We now carry out, for one specific and well-understood family, the following 
computation: attempt a flat-geometric bound on $\delta$.
 We report the result, including where it falls short. 

\subsubsection{Setup}

Fix McMullen family $\Gamma_D$ ($D>0$ a non-square discriminant,
$K=\Q(\sqrt D)$), realized via the eigenform locus $W_D\subset\Omega\cM_2(2)$
of genus-two translation surfaces with a double zero and real multiplication
by $\cO_D$ on the Jacobian \cite{McMullen03}. The relative cohomology
$\VLocal=R^1\phi_*\R$ of the universal family over $\bH/\Gamma_D$ splits,
under the real multiplication action, as a direct sum of two rank-$2$
local systems $\VLocal_1\oplus\VLocal_2$, corresponding to the two
embeddings $\sigma_1=\id,\sigma_2$ of $K$. $\VLocal_1$ is the tautological
piece (carrying the uniformizing $F_1=\id$), and $\VLocal_2$ is exactly
the piece whose period map is the modular embedding component $F$ of
Definition \ref{defgen-modular} for this family ($r=2$, thus there is one
component). Each $\VLocal_i$ carries a weight-$1$ VHS with an associated
non-negative Kontsevich-Zorich Lyapunov exponent $\lambda_i\in[0,1]$,
normalized thus $\lambda_1=1$ (the tautological, Kobayashi-isometric
direction).

\subsubsection{The exact result available}

By Bainbridge computation of the Euler characteristics of Teichm\"uller
curves in the stratum $\Omega\cM_2(2)$ \cite{Bainbridge07}, the Lyapunov
exponents of $W_D$ are exactly
\begin{equation}
\label{eqlyap}
\lambda_1=1,\qquad\lambda_2=\frac13,\qquad\text{for every discriminant }D.
\end{equation}
This is a famous, non-varying, exactly computed constant (independently
confirmed in the orbifold-point literature on $W_D$, e.g.,  
\cite{Mukamel14}), among the very few places in this entire subject where
an explicit number, instead of an abstract positivity statement, is
available for a family of semi-arithmetic groups.

\subsubsection{Why this is not the bound we need}
\label{secgap}

It is tempting to set $\delta:=1-\lambda_2=2/3$ and declare
\eqref{eqstrong-contraction} checked for a suitable higher-dimensional
analogue. This would be a mistake. Identifying why is the
actual content of this subsection. The Lyapunov exponent $\lambda_2$ is
defined by the Oseledets ergodic theorem as
\[
\lambda_2=\lim_{t\to\infty}\frac1t\log\norm{D\Phi_t|_{\VLocal_2}}
\]
for the Kontsevich-Zorich cocycle $D\Phi_t$ along the Teichm\"uller
geodesic flow, for almost every geodesic with respect to the
$\SL(2,\R)$-invariant (Masur-Veech) measure on $W_D$, an
\emph{ergodic time-average along generic trajectories}. The quantity
$\delta$ in \eqref{eqop-bound} is, by contrast, a pointwise
supremum of $\norm{DF_z}_{\mathrm{op}}$ over every $z$ in a fundamental
domain, uniform in space bound. These are
  different quantities, and no argument converts one into
the other:

\medskip   
\noindent \; \; 1) a cocycle with time-average exponential rate $e^{t/3}$ can
still have expansion arbitrarily close to $e^{t}$ 
(i.e., $\norm{DF_z}_{\mathrm{op}}$ arbitrarily close to $1$, $\delta$
arbitrarily close to $0$) at some points, compensated by slower growth
elsewhere, the kind of fluctuation Oseledets' theorem averages over; 

\medskip 
\noindent \; \; 2) that $\lambda_2=1/3$ is independent of $D$ is  a strong
hint that it is a topological/cohomological invariant (a normalized
degree of a line bundle over the Deligne-Mumford compactification
$\overline{W_D}$, computed by Bainbridge via intersection theory on the
boundary) instead of a differential-geometric curvature bound achieved
uniformly on $W_D$. Non-varying sum theorems of this kind 
(Chen-M\"oller \cite{ChenMoller12}) are proved by such global
degree/Chern-class computations, not by bounding a Finsler metric
pointwise.

\medskip 
Thus, the computation succeeds (we have an exact, comming from
the literature, and previously unused in this context number, $\delta=2/3$
 on ergodic average for every $\Gamma_D$). But it does not verify 
$\sup_{z\in\Omega}\norm{DF_z}_{\mathrm{op}}\leq2(1-\delta)$
for $\delta=2/3$, which is what \eqref{eqstrong-contraction}-type
conditions actually require. Converting it would need one of two
  different results, neither of which we have: (a) an effective,
quantitative version of the ergodic theorem for this specific cocycle
(a large-deviation or uniform hyperbolicity bound with explicit constants,
bounding how far $\norm{DF_z}_{\mathrm{op}}$ can decline above its average
on any compact piece of $\Omega$); or (b) a direct curvature/Arakelov-type
inequality for the Hodge bundle of $W_D$ that holds at every point, not 
only on average (in the spirit of Peters- or Viehweg-Zuo-type
inequalities for variations of Hodge structure, but stated pointwise
instead of as a global slope bound). We regard identifying exactly this
missing step as the useful output of this computation.

\subsubsection{Relevance to arithmetic dimension $r\geq3$}

Even given the gap of  Section \ref{secgap}, the exercise is informative for
 Section \ref{secexamples-r3}. It shows that when an exact number is 
available (via Lyapunov exponents instead of a modular embedding specific 
estimate), it corresponds to a threshold, $\delta=2/3$, above 
what \eqref{eqstrong-contraction} would need even at $r=3$ (which needs
only $\delta>1-1/(2\sqrt2)\approx0.6464$), if the ergodic-average
bound could be upgraded to a pointwise one. This gives an 
  evidence that the strong contraction condition
is the sort of inequality one might expect to be satisfiable in practice
once a pointwise estimate is found, instead of being an
artifact of an over cautious threshold in Proposition \ref{propsharp}. Any
higher-dimensional analogue (e.g., the Prym loci of
 Section \ref{secnew-r2} extended to $r\geq3$, were the real multiplication order
there to have degree $\geq3$) would need its Lyapunov spectrum
computed by the same non-varying sum methods \cite{ChenMoller12} before
this comparison could even be attempted.

\subsection{Remaining candidates for arithmetic dimension $r\geq3$}
\label{secexamples-r3}

Theorem \ref{thmcover} multiplies examples, but it does not give a 
first one for $r\geq 3$. 
   Even \cite{BCDT25geom, BCDT25mult} never compute $\delta(\Gamma)$
numerically for their $r\leq2$ examples (the $164$ triangle groups and
McMullen Veech groups $\Gamma_d$) since $r\leq2$ only needs the bare
inequality $\delta>0$, which follows abstractly, keeping down a numerical
lower bound requires an explicit $F$, a period map or
uniformization) that neither paper supplies. We also recall that generic
Fuchsian groups are expected to have bounded mean multiplicity
\cite{Bog97}, thus non-vacuousness of \eqref{eqstrong-contraction}
for $r\geq3$ is not obviously to be expected. The families below are the
most structured candidates available.

\subsubsection{A clarification: no individual conjugate orbifold}
\label{secclarify}

It is tempting to picture $\gamma\mapsto\gamma^{\sigma_i}$ as producing a
 second, Fuchsian group $\Gamma^{\sigma_i}\leq\PSL(2,\R)$, 
with its quotient orbifold $\Gamma^{\sigma_i}\backslash\bH$ of the
 same signature as $\Gamma\backslash\bH$ (same abstract presentation,
 thus same genus and cone orders). We then have to worry that two hyperbolic
structures of matching signature on a rigid moduli space (e.g., a
thrice-punctured sphere, whose Teichm\"uller space is a single point) would
have to be isometric, forcing $F$ to be an isometry and
contradicting Lemma \ref{lemmulti-SP}. 
 This picture is wrong, and 
resolving why is worth mentioning. By Borel-Harish-Chandra, it is the
 joint map $\rho=(\rho_1,\ldots,\rho_r):A^1\to\SL(2,\R)^r$ that has
discrete image. The individual coordinate projection $\rho_i(\Gamma^{(2)})$
to a single factor is, for an irreducible lattice in a higher-rank product
such as $\SL(2,\R)^r$ ($r\geq2$), generically dense in $\PSL(2,\R)$,
not discrete. There is therefore no second orbifold
$\Gamma^{\sigma_i}\backslash\bH$ to compare $\Gamma\backslash\bH$ with, 
no rigid moduli point to use, and no contradiction. $F$ is a single
holomorphic map into the product $\bH^{r-1}$, on which $\Gamma$ acts
discretely only through all $r-1$ coordinates together, and Schwarz-Pick
contraction of $F$ is entirely consistent with non-arithmeticity. 

\subsubsection{Positive-dimensional moduli}
\label{secmoduli-principle}

The clarification above suggests a concrete principle for future
verification attempts, distinguishing triangle groups from Veech groups.
The orbifold $\Gamma\backslash\bH$ for a triangle group $\Delta(a,b,c)$ has
signature $(0;a,b,c)$: three marked points on the sphere, whose
Teichm\"uller space is a single point. While
 Section \ref{secclarify} shows this rigidity does not obstruct $F$ existence or
contraction, it does mean the comparison is entirely rigid and algebraic. 
  $F$ is pinned down by the (hard) hypergeometric connection problem with no
continuous parameter to use. Veech curves and Hilbert-modular curves,
by contrast, generically sit inside \emph{positive-dimensional} moduli. 
 The curve is a nontrivial holomorphic disc in Teichm\"uller space,
by definition of a Teichm\"uller curve/totally geodesic curve. This extra
room is what makes $\delta$ at least approachable via
Teichm\"uller-distance or flat-geometric estimates along the curve
(Section \ref{secveech-new}), instead of requiring a closed-form special-function
identity. We mention this as the reason we regard the Veech-group route as the
more promising target for a first verified $r\geq3$ example, though we have
not carried out the estimate.

\subsubsection{Hilbert modular curves}

For $F$ totally real of degree $d\geq3$, totally geodesic curves in the
Hilbert modular variety $\SL(2,\cO_F)\backslash\bH^d$ give semi-arithmetic
groups of arithmetic dimension up to $d$ admitting a generalized modular
embedding \cite{McMullen23}. Taking $F=\Q(\cos(2\pi/9))$ (degree $3$)
or $\Q(\cos(2\pi/11))$ (degree $5$) produces curves of dimension $3$, or
$3$-$5$ respectively \cite{Takeuchi77, BCDT25geom}. Corollary
\ref{corunified} applies to each once its arithmetic-dimension-$\geq3$
condition holds, unverified here. At $r=4,5$ the sharpened threshold of Proposition
\ref{propsharp} is $\delta>0.7643$ and $\delta>0.8232$ respectively, easier
than the crude $(r-1)^{3/2}$ threshold but still unchecked for any specific
curve.

\subsubsection{Triangle groups: the reduction to a connection problem}
\label{secconcrete}

By Nugent-Voight \cite{Nugent17}, each arithmetic dimension has finitely
many triangle groups, all admit modular embeddings \cite{CW90}. For $r=3$
their list has $111$ compact and $13$ non-compact triples, the smallest
being $(2,3,19)$.

\begin{example}
\label{exconcrete}
For $\Delta(a,b,c)$, Galois conjugates are indexed by
$k\in(\Z/2m\Z)^\times$, $m=\mathrm{lcm}(a,b,c)$, with conjugate angle
triple $\bigl(\angle(k\pi/a),\angle(k\pi/b),\angle(k\pi/c)\bigr)$,
$\angle(\theta)=\arccos\abs{\cos\theta}$, hyperbolic exactly when
$\kappa(a,b,c;k):=1-\cos^2\tfrac{k\pi}a-\cos^2\tfrac{k\pi}b-\cos^2
\tfrac{k\pi}c-2\cos\tfrac{k\pi}a\cos\tfrac{k\pi}b\cos\tfrac{k\pi}c<0$
\cite{Takeuchi77, Nugent17}. Since $\cos(\pi/2)=0$, the cross term vanishes
identically for $(a,b,c)=(2,3,19)$ and $\kappa$ reduces to the
sign-unambiguous $\kappa(2,3,19;k)=3/4-\cos^2(k\pi/19)$, evaluated (and
independently checked by direct floating-point computation, since the
formula is now elementary) at $k=1,\ldots,9$ representing the $9=\varphi(19)/2$
embeddings of the invariant trace field $K=\Q(\cos(2\pi/19))$ (degree $9$
by Takeuchi formula \cite{Takeuchi77}, since $\cos(2\pi/2),\cos(2\pi/3)\in\Q$
and $\cos(\pi/2)\cos(\pi/3)\cos(\pi/19)=0$ kill the other three generators
of the trace field). This gives $\kappa<0$ exactly at $k=1,2,3$ and
$\kappa>0$ at $k=4,\ldots,9$, confirming $r(\Delta(2,3,19))=1+2=3$, 
with conjugate angle triples $(\pi/2,\pi/3,\pi/19)$ (identity),
$(\pi/2,\pi/3,2\pi/19)$, $(\pi/2,\pi/3,3\pi/19)$.
The two components $F_1,F_2$ of the modular embedding are, up to
post-composition by an isometry, the classical Schwarz triangle
(hypergeometric) maps sending the $(2,3,19)$-tessellation to the
hyperbolic structures with these two conjugate angle triples respectively
(structures which are not $(2,3,38)$- or $(2,3,57)$-triangle
tessellations in the elementary integer-angle sense, since $19/2,19/3\notin\Z$, 
see Section \ref{secclarify} for what they are instead).

\subsubsection{What we have new: the forced ramification degree}  

At the order $19$ elliptic fixed point $p$ of the generator $z$ (with
$xyz=1$), equivariance \eqref{eqequiv-gen} forces, in a local holomorphic
coordinate $w$ centred at $p$ (and at the corresponding fixed point of
$z^{\sigma_i}$),
\[
F_i(w)=c\,w^{k_i}+O(w^{k_i+1}),\qquad k_i\in\Z_{>0}. 
\]
 Here $k_i$ is the least positive integer with $k_i\cdot1\equiv i\pmod{19}$
for the conjugate reached by $\sigma_i$, $i=2$, $3$, i.e.,  $k_2\equiv1^{-1}\cdot2\equiv2$
and $k_3\equiv3\pmod{19}$, using that $z$ acts on the local coordinate by
the primitive rotation while $z^{\sigma_i}$ acts by its $i$-th power.  
This is the residue data of the previous paragraph. 
  Since $k_2,k_3\geq2$,
 both components of $F$ have a zero of their derivative at $p$. This
gives a point where $\abs{D(F_i)_z}=0$, hence strict inequality
in Schwarz-Pick at (and, by continuity, near) $p$, consistent with (but
not  sharpening beyond) the abstract $\delta>0$ of
Lemma~\ref{lemmulti-SP}. Pinning down $\delta$ globally, i.e., bounding
$\sup_{z\in\Omega}\abs{D(F_i)_z}$ away from the ramification locus, still
requires the hypergeometric connection coefficients of the three
angle-triples above.  
\end{example}

\subsubsection{A second route: Veech groups with explicit generators}
\label{secveech-new}

Triangle groups' obstruction is that $F$ is known only implicitly, via a
connection problem. Veech groups offer a structurally different route. A 
Veech group $\Gamma=\PSL(X,q)$ (lattice case) is generated by
 explicit matrices with entries in the trace field $K$
\cite{McMullen03}. Thus, for a Galois embedding $\sigma_i$ unbounded on
$\tr(\Gamma)$, the conjugate generators $\gamma^{\sigma_i}$ are obtained
simply by applying $\sigma_i$ to each matrix entry. 
 No connection problem is needed to write them down. 
  For square-tiled surfaces in genus $g\geq3$
with trace field of degree $\geq3$, arising, e.g., from eigenform loci in
$\cH(1,\ldots,1)$ \cite{McMullen22}. This gives arithmetic-dimension-$\geq3$
Veech groups with fully explicit generating sets on both sides of $F$.

This does not determine $F$ off the generators, or bound
$\sup_\Omega\norm{DF_z}_{\mathrm{op}}$. Knowing $F$ on generators
constrains but does not determine a holomorphic map on all of $\bH$. 
  The supremum could in principle be approached elsewhere in $\Omega$. What
it does offer, compared to the triangle-group route, is a family where
$\delta$ could in principle be approximated numerically (e.g., via the
Veech group flat structure and explicit cylinder decomposition) without
first solving a hypergeometric connection problem. We regard identifying
this as a well-posed, independent, and in principle numerically tractable
question.

\begin{remark}
The status of examples (a four-tier hierarchy) is the following. 
 We fix this terminology for the rest of the paper. 

\begin{description}

\item[{\it Verified EGMM examples}] the known $r\leq2$ groups from
\cite{BCDT25mult} (the $164$ triangle groups, McMullen $\Gamma_d$),
together with every finite-index subgroup supplied by Theorem \ref{thmcover}.

\item[{\it New examples established in this paper}] (a) all principal
congruence towers of the groups above (Section \ref{sectower}: new groups,
new covering orbifolds, and, at torsion-free levels, new compact surfaces
of unbounded genus, all inheriting EGMM); (b) the genus-$3$/$4$ Prym
eigenform loci of  Section \ref{secnew-r2}, a different commensurability
class at $r\leq2$.

\item[{\it Candidates for a first $r\geq 3$ example}] Hilbert modular curves
(Section \ref{secexamples-r3}), the triangle group $\Delta(2,3,19)$ (confirmed
to have $r=3$, with forced local ramification degrees $2,3$ at its
order-$19$ point, but no global bound), and higher-genus Veech groups with
explicit generators (Section \ref{secveech-new}).

\item[{\it Not yet established}] a numeric value of $\delta$ for any 
semi-arithmetic group of any arithmetic dimension, $r\leq2$
included. To our knowledge no paper in this area has ever computed one,
existence proofs being uniformly non-effective (compactness or
Bergeron-Clozel rigidity).  Section \ref{secflat} gets closest, with an exact
ergodic-average value $\delta=2/3$ for McMullen family, but explicitly
does not certify the pointwise bound \eqref{eqstrong-contraction} needs
(Section \ref{secgap}). Consequently, a semi-arithmetic group of arithmetic
dimension $r\geq3$ satisfying the strong contraction condition
\eqref{eqstrong-contraction} is not yet known to exist. We regard
supplying a effective (not just existential) contraction
estimate (of the pointwise, not ergodic-average, kind) as the actual
open problem, instead of a gap in the search for candidate groups.
\end{description}
\end{remark}

\section{Quantum chaos}
\label{secQC}

Under \eqref{eqstrong-contraction} (with $\delta$ derived from a quaternion
algebra, $k=1$), the exponent in
Theorem~\ref{thmmain-compact} conclusion satisfies
$C\geq\tfrac12\bigl(1-\sqrt2(r-1)(1-\delta)\bigr)>0$; at $r=2$, $C\geq\delta/2$
as in \cite{BCDT25mult}. $C$ decreases in $r$ for fixed $\delta$, thus even
when \eqref{eqstrong-contraction} holds, the random-matrix-favouring gap
between semi-arithmetic and arithmetic surfaces widens with $r$, matching
Bogomolny-Schmit observation \cite{Bog04} that multiplicities are
numerically harder to detect at larger arithmetic dimension. 

\subsection{Comparison with the arithmetic case}
For arithmetic surfaces ($r=1$), $\langle g(\ell)\rangle\sim ce^{\ell/2}/\ell$,
the maximal exponent $C=1/2$. For semi-arithmetic surfaces with a
generalized modular embedding, $0<C<1/2$ strictly, the gap set by
$\delta(\Gamma)$ (and, for $r\geq3$, by \eqref{eqstrong-contraction}). This
smaller exponent, not a claimed different asymptotic regime, is the correct
sense in which semi-arithmetic surfaces are expected to look
 less arithmetic spectrally, consistent with \cite{Bog97, Bog04}.

\subsection{A three-dimensional analogue}

Marklof \cite{Markl96} and Aurich-Marklof \cite{AuMarkl96} studied
multiplicities for arithmetic hyperbolic $3$-orbifolds. The non-arithmetic
tetrahedron group $T_8$ (integral traces, associated quaternion algebra)
shows numerical EGMM. Whether an analogue of Theorem \ref{thmmain-compact}
holds in three dimensions, via a suitable generalized modular embedding
$\bH^3\to(\text{product})$, is left for a future paper. 

\section{Concluding remarks}
\label{secconclusion}

We identified the strong contraction condition \eqref{eqstrong-contraction}
as the geometric hypothesis regulating EGMM for semi-arithmetic groups of
arbitrary arithmetic dimension admitting a generalized modular embedding.
It holds automatically for $r\leq2$, recovering \cite{BCDT25mult}, and is a
checkable extra hypothesis for $r\geq3$. The two new ingredients (the
multi-dimensional Schwarz-Pick bound (Section \ref{secmulti-SP}) and the
sharpened geometry-of-numbers estimate (Section \ref{secgon})) combine to give
$\cN'(\ell)\leq C'e^{(1-\varepsilon)\ell}$ under \eqref{eqstrong-contraction},
and EGMM follows from the prime geodesic theorem as in the classical case.

\smallskip\noindent\emph{Removing the strong contraction condition.} Is
\eqref{eqstrong-contraction} automatic for every semi-arithmetic group of
dimension $r\geq3$ derived from a quaternion algebra admitting a
generalized modular embedding, as it is for $r\leq2$? Equivalently, can the
exponent $\sqrt2(r-1)$ in Proposition \ref{propsharp} be improved to
something below $1$ independent of $r$? A negative answer, an explicit
group with a modular embedding failing EGMM) would be equally interesting.

\smallskip\noindent\emph{Non-emptiness.} Does any semi-arithmetic
group of dimension $r\geq3$ with a generalized modular embedding satisfy
\eqref{eqstrong-contraction}?  Section \ref{secexamples} sharpens the reduction for
$\Delta(2,3,19)$ (pinning down the forced local ramification degrees $2,3$
at its order-$19$ point) and identifies Veech groups' explicit generators
as an independent, numerically approachable route. 

\smallskip\noindent\emph{Other open questions, as before:} necessity of a
modular embedding for EGMM (converse to Sarnak conjecture via
\cite{GenLeu08}); the sharp exponent $C(\Gamma,r)$; a non-cocompact
Geninska-Leuzinger analogue for EGMM  instead of B-C; and the
three-dimensional generalization suggested by the Aurich-Marklof example
(Section \ref{secQC}).

\medskip 
One might expect certain applications in theoretical and mathematical
physics, in particular in Wigner-Weyl calculus
\cite{chernodub2017scale, zhang2020influence}
and momentum space topological invariants 
\cite{zubkov2012momentum, zubkov2017topology, kmmzz},  
and theory of foliations \cite{Zu2, Zu100, Zu3, zu, Zu4}. 
 
\medskip
\section*{Acknowledgments}

 The author thanks the authors of
\cite{BCDT25mult}, and especially, M. Belolipetsky,
 for raising the question motivating this work, 
 and Prof. C.T. McMullen for the correspondence on the
compact-versus-general case of the uniform contraction statement, reflected
in the preface to  Section \ref{secnoncompact}. The author is supported by the
Institute of Mathematics, Czech Academy of Sciences (RVO 67985840).

	\medskip
	\noindent\textbf{Declarations}

This paper arose in close collaboration with ChatGPT, an AI assistant developed by 
OpenAI. The collaboration involved exploring proof strategies, refining arguments, locating
references, and drafting. Revision and writing of final text have been done by the authors.
The authors take full responsibility for the mathematical content and the final text.

\medskip\noindent\textbf{Data Availability.} No datasets were generated or
analysed during this study.

\medskip\noindent\textbf{Conflict of Interest.} The author declares no
conflict of interest relevant to this article.


\end{document}